\documentclass{amsart}

\usepackage{amsmath,amsthm,amsfonts, amssymb,amscd}
\usepackage[all]{xy}

\newtheorem{theorem}{Theorem}[section]
\newtheorem{lemma}[theorem]{Lemma}
\newtheorem{example}[theorem]{Example}
\newtheorem{remark}[theorem]{Remark}
\newtheorem{corollary}[theorem]{Corollary}
\newtheorem{proposition}[theorem]{Proposition}

\newcommand{\minusre}{\hspace{0.3em}\raisebox{0.3ex}{\sl \tiny /}\hspace{0.3em}}
\newcommand{\minusli}{\hspace{0.3em}\raisebox{0.3ex}{\sl \tiny $\setminus $}\hspace{0.3em}}

\newcommand{\dx}{\, \mbox{\rm d}}

\newcommand{\Aff}{\mbox{\rm Aff}}

\begin{document}
\title[The  Structure of States on Pseudo Effect Algebras]{The
Lattice and Simplex Structure of States on Pseudo Effect Algebras}
\author{Anatolij Dvure\v censkij}
\date{}
\maketitle

\begin{center}
\footnote{Keywords: Pseudo effect algebra; effect algebra; Riesz
Decomposition Properties; signed measure, state; Jordan signed
measure, unital po-group; simplex; Choquet simplex; Bauer simplex

AMS classification:  81P15, 03G12, 03B50

The  author thanks  for the support by Center of Excellence SAS
-~Quantum Technologies~-,  ERDF OP R\&D Projects CE QUTE ITMS
26240120009 and meta-QUTE ITMS 26240120022, the grant VEGA No.
2/0032/09 SAV. }
\small{Mathematical Institute,  Slovak Academy of Sciences\\
\v Stef\'anikova 49, SK-814 73 Bratislava, Slovakia\\
E-mail: {\tt dvurecen@mat.savba.sk}, }
\end{center}

\begin{abstract}  We study states, measures, and signed measures on
pseudo effect algebras with some kind of the Riesz Decomposition
Property, (RDP). We show that the set of all Jordan signed measures
is always an Abelian Dedekind complete $\ell$-group. Therefore, the
state space of the pseudo effect algebra with (RDP) is either empty
or a nonempty Choquet simplex or even a Bauer simplex. This will
allow represent states on pseudo effect algebras by standard
integrals.
\end{abstract}

\section{Introduction}

The seminal paper by Birkhoff and von Neumann \cite{BiNe} showed
that the events of quantum mechanical measurements do not fulfill
the axioms of Boolean algebras and therefore also do not axioms of
the classical probability theory presented by Kolmogorov \cite{Kol}.
It initiated the research of the mathematical foundations of quantum
physics. Nowadays, there appeared a whole hierarchy of so-called
quantum structures, like orthomodular lattices and posets,
orthoalgebras, etc. Since the Nineties, we are intensively studying
effect algebras that were introduced by Foulis and Bennett
\cite{FoBe}. An extensive source of information about effect
algebras can be found in \cite{DvPu}. Orthodox examples  of the
Hilbert space quantum mechanics are the system of closed subspaces,
$\mathcal L(H),$  of a Hilbert space $H$ (real, complex or
quaternionic) and the system of all Hermitian operators, $\mathcal
E(H),$ that are between the zero operator and the identity operator.
An effect algebra $E$ is a partial algebraic structure with a
partially defined binary operation, $+$, that is commutative and it
models join of ``mutually exclusive" events. In many cases, it is an
interval in a po-group (= partially ordered group), like $\mathcal
E(H)$ is an interval in the po-group $\mathcal B(H)$ of all
Hermitian operators on a Hilbert space $H.$ A sufficient condition
for an effect algebra to be an interval is e.g. the Riesz
Decomposition Property (RDP, for short); and in such a case, $E$ is
an interval in a unique unital Abelian po-group $(G,u)$ with
interpolation, or equivalently, with (RDP), see \cite{Rav} or
\cite[Thm 1.7.17]{DvPu}.

In the last decade, there appeared many structures where the basic
operation, $+,$ is not necessarily commutative. The papers
\cite{DvVe1, DvVe2} present a non-commutative generalization of
effect algebras, called {\it pseudo effect algebras}. In some
important examples, they are also an interval in a unital po-group
but not necessarily Abelian. Sufficient conditions for a pseudo
effect algebra to be an interval in a unital po-group are stronger
versions of (RDP), see \cite{DvVe2} for more details.

Any measurement is accomplished by probabilistic reasoning. The
quantum mechanical one is described by a state, an analogue of a
probability measure.  The state space of any pseudo effect algebra
is an interesting structure that can be also void, see e.g.
\cite{Dvu1}, but in general it is a convex compact Hausdorff
topological space. In very important cases, it is a simplex and this
allows then characterize states via an integral through a regular
Borel probability measure, in some cases even in a unique way, see
\cite{Dvu3}.

If an effect algebra satisfies (RDP), then it is an interval in an
Abelian unital po-group with interpolation (RIP), so that it is  a
non-void simplex, \cite[Thm 5.1]{Dvu2}. If $E$ is a pseudo effect
with (RDP) that is an interval in a unital po-group, then it can
happen that the state space is empty, \cite{Dvu1}. In \cite[Thm
4.2]{Dvu3}, we have showed that every interval pseudo effect algebra
with (RDP) or an effect algebra with (RDP)$_1$ is a simplex.

The Riesz Decomposition Property is a weaker form of
distributivity~- it allows to make a joint refinement of two
decompositions of the unit element. This is a reason why (RDP) fails
to hold for $\mathcal L(H)$ and $\mathcal E(H).$

We do not know whether every pseudo effect algebra with (RDP) is an
interval in a unital po-group, this is known only for a stronger
version (RDP)$_1$, \cite[Thm 5.7]{DvVe2}. Hence, we cannot directly
apply the result from \cite[Thm 4.2]{Dvu3}. Therefore, we prove in
the paper that the state space of a pseudo effect algebra with (RDP)
is empty or a non-void Choquet simplex, Theorem \ref{th:4.2}.  To
prove that, we are studying the set of Jordan signed measures on a
pseudo effect algebra with (RDP). We show that such a set is either
a singleton containing only the zero measure or it is a non-trivial
Abelian Dedekind complete $\ell$-group (= lattice ordered). The
simplex structure will be finally applied to represent a state as an
integral through a unique regular Borel probability measure. We note
that such a representation of states for MV-algebras (= effect
algebras with (RDP)$_2$ = Phi-symmetric effect algebras, see
\cite{BeFo}) was proved in \cite{Kro, Pan} and for effect algebras
in \cite{Dvu3}.

The paper is organized as follows.

The elements of pseudo effect algebras are presented in Section 2.
Section 3 describes the lattice structure of the group of all Jordan
signed measures on a pseudo effect algebra with (RDP). Section 4
will describe some basic properties of Jordan signed measures that
were known only for classical measures. Applications of the simplex
structures of the state space, Choquet or Bauer simplices, for
representation of states by integral are given in Section 5. The
final concluding remarks are presented in Section 6.

\section{Pseudo Effect Algebras}

Following \cite{DvVe1,Dvu2}, we say that a {\it pseudo effect
algebra} is a partial algebra  $(E; +, 0, 1)$, where $+$ is a
partial binary operation and $0$ and $1$ are constants, such that
for all $a, b, c \in E$, the following holds

\begin{enumerate}
\item[(i)] $a+b$ and $(a+b)+c$ exist if and only if $b+c$ and
$a+(b+c)$ exist, and in this case $(a+b)+c = a+(b+c)$;

\item[(ii)]
  there is exactly one $d \in E$ and
exactly one $e \in E$ such that $a+d = e+a = 1$;

\item[(iii)]
 if $a+b$ exists, there are elements $d, e
\in E$ such that $a+b = d+a = b+e$;

\item[(iv)] if $1+a$ or $a+1$ exists, then $a = 0$.
\end{enumerate}

If we define $a \le b$ if and only if there exists an element $c\in
E$ such that $a+c =b,$ then $\le$ is a partial ordering on $E$ such
that $0 \le a \le 1$ for any $a \in E.$ It is possible to show that
$a \le b$ if and only if $b = a+c = d+a$ for some $c,d \in E$. We
write $c = a \minusre b$ and $d = b \minusli a.$ Then

$$ (b \minusli a) + a = a + (a \minusre b) = b,
$$
and we write $a^- = 1 \minusli a$ and $a^\sim = a\minusre 1$ for any
$a \in E.$

For basic properties of pseudo effect algebras see \cite{DvVe1,
DvVe2}. We recall that if $+$ is commutative, $E$ is said to be an
{\it effect algebra}; for a comprehensive overview on effect
algebras see e.g. \cite{DvPu}. It is worthy to remark that effect
algebras are equivalent to D-posets, where the basic operation is a
difference of two comparable events, \cite{KoCh}.

We recall that a {\it po-group} (= partially ordered group) is a
group $G$ with a partial order, $\le,$ such that if $a\le b,$ $a,b
\in G,$ then $x+a+y \le x+b+y$ for all $x,y \in G.$  We denote by
$G^+$ the set of all positive elements of $G.$ If, in addition,
$\le$ implies that $G$ is a lattice, we call it an $\ell$-group (=
lattice ordered group). An element $u\in G^+$ is said to a {\it
strong unit} if given $g \in G,$ there is an integer $n\ge 1$ such
that $g \le nu,$ and the couple $(G,u)$ with a fixed strong unit is
said to  a {\it unital po-group.} The monographs like \cite{Fuc,
Gla} can serve as a basic source of information about partially
ordered groups.

If $(G,u)$ is a unital (not necessary Abelian) po-group with strong
unit $u$, and
$$
\Gamma(G,u) := \{g \in G: \ 0 \le g \le u\},\eqno(2.1)
$$
then $(\Gamma(G,u); +,0,u)$ is a pseudo effect algebra if we
restrict the group addition $+$ to  the set of all those $(x,y)\in
\Gamma(G,u)\times \Gamma(G,u)$ that $x\le u-y.$

Every pseudo effect algebra $E$ that is isomorphic to some
$\Gamma(G,u)$ is said to be an {\it interval pseudo effect algebra}.

According to \cite{DvVe1}, we introduce for pseudo effect algebras
the following forms of the Riesz Decomposition Properties which in
the case of commutative effect algebras can coincide:

\begin{enumerate}
 \item[(a)] For $a, b \in E$, we write $a \ \mbox{\bf com}\ b$ to mean that
for all $a_1 \leq a$ and $b_1 \leq b$, $a_1$ and $b_1$ commute.

\item[(b)] We say that $E$ fulfils the {\it Riesz
Interpolation Property}, (RIP) for short, if for any $a_1, a_2, b_1,
b_2 \in E$ such that $a_1, a_2 \,\leq\, b_1, b_2$ there is a $c \in
E$ such that $a_1, a_2 \,\leq\, c \,\leq\, b_1, b_2$.

 \item[(c)] We say that $E$ fulfils the {\it weak Riesz
Decomposition Property}, (RDP$_0$) for short, if for any $a, b_1,
b_2 \in E$ such that $a \leq b_1+b_2$ there are $d_1, d_2 \in E$
such that $d_1 \leq b_1$, $\;d_2 \leq b_2$ and $a = d_1+d_2$.

 \item[(d)] We say that $E$ fulfils the {\it Riesz
Decomposition Property}, (RDP) for short, if for any $a_1, a_2, b_1,
b_2 \in E$ such that $a_1+a_2 = b_1+b_2$ there are $d_1, d_2, d_3,
d_4 \in E$ such that $d_1+d_2 = a_1$, $\,d_3+d_4 = a_2$, $\,d_1+d_3
= b_1$, $\,d_2+d_4 = b_2$.

 \item[(e)] We say that $E$ fulfils the {\it
commutational Riesz Decomposition Property}, (RDP$_1$) for short, if
for any $a_1, a_2, b_1, b_2 \in E$ such that $a_1+a_2 = b_1+b_2$
there are $d_1, d_2, d_3, d_4 \in E$ such that (i) $d_1+d_2 = a_1$,
$\,d_3+d_4 = a_2$, $\,d_1+d_3 = b_1$, $\,d_2+d_4 = b_2$, and (ii)
$d_2 \ \mbox{\bf com}\ d_3$.

 \item[(f)] We say that $E$ fulfils the {\it strong
Riesz Decomposition Property}, (RDP$_2$) for short, if for any $a_1,
a_2, b_1, b_2 \in E$ such that $a_1+a_2 = b_1+b_2$ there are $d_1,
d_2, d_3, d_4 \in E$ such that (i) $d_1+d_2 = a_1$, $\,d_3+d_4 =
a_2$, $\,d_1+d_3 = b_1$, $\,d_2+d_4 = b_2$, and (ii) $d_2 \wedge d_3
= 0$.
\end{enumerate}

We have the implications

\centerline{ (RDP$_2$) $\Rightarrow$ (RDP$_1$) $\Rightarrow$ (RDP)
$\Rightarrow$\ (RDP$_0$) $\Rightarrow$ (RIP). }

The converse of any of these implications does not hold. For
commutative effect algebras we have

\centerline{ (RDP$_2$) $\Rightarrow$ (RDP$_1$) $\Leftrightarrow$
(RDP) $\Leftrightarrow$\ (RDP$_0$) $\Rightarrow$ (RIP). }

In addition, every pseudo effect algebra with (RDP)$_2$ is an
interval in a unital $\ell$-group, \cite[Prop 3.3]{DvVe1}.

In an analogous way we can define the same Riesz Decomposition
Properties for a po-group $G$, where instead of $E$ we deal with the
positive cone $G^+.$

We recall that an {\it MV-algebra} is an algebra $(A;\oplus,^*,0)$
of signature $\langle 2,1,0\rangle,$ where $(A;\oplus,0)$ is a
commutative monoid with neutral element $0$, and for all $x,y \in A$
\begin{enumerate}
\item[(i)]  $(x^*)^*=x,$
\item[(ii)] $x\oplus 1 = 1,$ where $1=0^*,$
\item[(iii)] $x\oplus (x\oplus y^*)^* = y\oplus (y\oplus x^*)^*.$
\end{enumerate}

Sometimes it is used also a total binary operation $\odot$ defined
by $a\odot b:= (a^*\oplus b^*)^*.$

If we define a partial addition, $+$, via $a+b$ is defined iff $a\le
b^*$, then $a+b = a\oplus b,$ then $(A;+,0,1)$ is an effect algebra
with (RDP)$_2,$ \cite{DvVe2}, or equivalently a Phi-symmetric effect
algebra, \cite{BeFo}; and it is an interval in an Abelian unital
$\ell$-group. Conversely, every lattice ordered effect algebra with
(RDP) or equivalently, every effect algebra with (RDP)$_2$ is in
fact an MV-algebra.

\section{Signed Measures and Jordan Signed Measures on Pseudo Effect Algebras} 

In the present section, we describe the lattice structures of the
set of Jordan signed measures on a pseudo effect algebra satisfying
(RDP). We show that it is either  trivial or a nontrivial Dedekind
complete Riesz space.

Let $E$ be a pseudo effect algebra. A {\it signed measure} on $E$ is
any mapping $m: E\to \mathbb R$ such that $m(a+b)=m(a)+m(b)$
whenever $a+b$ is defined in $E$.  Then $m(0)=0$ and  $m(a^-)=
m(a^\sim)$ for each $a\in E.$ A {\it measure} is a positive signed
measure $m,$ i.e. $m(a)\ge 0$ for $a \in E.$ Every measure is
monotone on $E.$  A {\it state} on $E$ is any measure $s$ such that
$s(1)=1.$ Let $\mathcal M(E),$ $\mathcal M(E)^+,$ and $\mathcal
S(E)$ be the sets of all signed measures,   measures, and states on
$E,$ respectively. It is clear that $\mathcal M(E)\ne \emptyset$
whilst $\mathcal S(E)$ can be empty. On $\mathcal M(E)$ we introduce
a {\it weak topology} of signed measures defined as follows: a net
of signed measures, $\{m_\alpha\},$ converges weakly to a signed
measure $m$ iff $\lim_\alpha m_\alpha(a)= m(a)$ for every $a \in E.$
Then $\mathcal M(E)$ is a non-void compact Hausdorff topological
space. Similarly,  $\mathcal S(E)$ is a compact Hausdorff space that
can be sometimes void. Moreover, $\mathcal S(E)$ is a convex set,
i.e. if $s_1,s_2 \in \mathcal S(E)$ and $\lambda \in [0,1],$ then $s
= \lambda s_1 +(1-\lambda) s_2 \in \mathcal S(E).$ A state $s$ is
{\it extremal} if from the property $s = \lambda s_1 +(1-\lambda)
s_2$ for some $s_1,s_2 \in \mathcal S(E)$ and $\lambda \in (0,1),$
we conclude $s= s_1 =s_2.$ Let $\partial_e \mathcal S(E)$ denote the
set of all extremal states on $E.$

By the Krein--Mil'man Theorem,  \cite[Thm 5.17]{Goo}, every state on
$E$ is a weak limit of a net of convex combinations of extremal
states. Hence, $\mathcal S(E) \ne \emptyset$ iff $\partial_e
\mathcal S(E)\ne \emptyset.$

In what follows, we are inspired by the research in \cite[pp.
37-41]{Goo}, where it was done for Abelian po-groups.

A mapping $d:E\to \mathbb R$ is said to be {\it subadditive}
provided $d(0) = 0$ and $d(x+y)\le d(x)+d(y)$ whenever  $x+y \in E.$

\begin{proposition}\label{pr:3.1}  Let $E$ be a pseudo effect algebra with {\rm
(RDP)} and let $d:E\to \mathbb R$ be a subadditive mapping. For all
$x\in E$, assume that the set
$$
D(x):=\{d(x_1)+\cdots+d(x_n): x = x_1+\cdots+x_n, \ x_1,\ldots,x_n
\in E,\ n\ge 1 \} \eqno(3.1)
$$
is bounded above in $\mathbb R.$  Then there is a signed measure
$m:E\to \mathbb R$ such that $m(x)=\bigvee D(x)$ for all $x\in E.$

\end{proposition}

\begin{proof}
The map $m(x):=\bigvee D(x)$ is a well-defined mapping for all $x
\in E.$ It is clear that $m(0)=0$ and now we are  going to show that
$m$ is additive on $E.$

Let $x+y \in E$ be given. For all decompositions
$$ x = x_1 +\cdots+x_n \ \mbox{and} \ y=y_1+\cdots +y_k$$
with all $x_i,y_j \in E,$ we have $x+y = x_1+\cdots+x_n + y_1+\cdots
+ y_k,$ that yields

$$ \sum_i d(x_i)+\sum_j d(y_j) \le m(x+y).
$$
Therefore, $u+v \le m(x+y)$ for all $u\in D(x)$ and $b\in D(y).$
Since $\mathbb R$ is Dedekind complete, $\bigvee$ is distributive
with respect to $+:$
\begin{eqnarray*}
m(x)+m(y)&=& \left(\bigvee D(x)\right) +m(y) = \bigvee_{u \in D(x)} (u+m(y))\\
&=& \bigvee_{u\in D(x)} \left(u+ \left(\bigvee D(y)\right)\right) =
\bigvee_{u\in
D(x)}\bigvee_{v \in D(y)} (u+v)\\
&\le& m(x+y).
\end{eqnarray*}

Conversely, let $x+y=z_1+\cdots+z_n,$ where each $z_i \in E.$ Then
(RDP) implies that there are elements $x_1,\ldots,x_n, y_1,\ldots,
y_n \in E$ such that $x = x_1+\cdots+x_n,$ $y = y_1+\cdots+y_n$ and
$z_i = x_i+y_i$ for $i=1,\dots,n.$ This yields
$$
\sum_i d(z_i) \le \sum_i (d(x_i)+d(y_i)) = \left(\sum_i
d(x_i)\right) + \left(\sum_i d(y_i)\right) \le m(x)+m(y),
$$
and therefore, $m(x+y)\le m(x)+m(y)$ and finally, $m(x+y)=m(x)+m(y)$
for all $x,y \in E$ such that $x+y$ is defined in $E,$ so that $m$
is a signed measure on $E.$
\end{proof}

Let $X$  be a poset. A mapping $m:X \to \mathbb R$ is said to be (i)
{\it relatively bounded} provided that given any subset $W$ of $X$
which is bounded (above and below) in $X$, the set $m(W)$ is bounded
in $\mathbb R,$ (ii) {\it bounded} if $m(X)$ is bounded in $\mathbb
R.$

We recall that if $m$ is a signed measure on  $E,$ then $m$ is
relatively bounded iff $m$ is bounded.

If $G$ is a  po-group, any group homomorphism $m:G \to \mathbb R$ is
said to be a {\it signed measure} on $G.$ Of course, if $m\ne 0$ is
a  measure  that is relatively bounded on $G\ne \{0\},$ then it is
not bounded on $G.$

\begin{lemma}\label{le:4.0}
If $m$ is a signed measure on a unital po-group $(G,u),$ then $m$ is
relatively bounded iff $m$ is bounded on the interval $[0,nu]$ for
each $n\ge 1.$  If, in addition, $(G,u)$ satisfies {\rm (RDP)}, then
$m$ is relatively bounded iff $m$ is bounded on the interval
$[0,u].$
\end{lemma}

\begin{proof}
Indeed, one direction is clear, now suppose that $m$ is bounded on
each interval $[0,nu],$ and let $W$ be bounded in $G.$  Then $W
\subseteq [a,b]$ and for some $a,b\in G.$ There is an integer $n\ge
1$ such that $-a+b \le nu.$ Then $[a,b] = a +[0,-a+b] \subseteq a +
[0,nu]$ and  $m(W)\subseteq m(a) + m([0,nu]) \subseteq m(a) +
[\alpha,\beta] = [m(a)+\alpha, m(a)+\beta]$ for some
$\alpha,\beta\in \mathbb R.$  This gives $m(W)$ is bounded in
$\mathbb R.$

If, in addition, $(G,u)$ satisfies (RDP), then $[0,nu] =
[0,u]+\cdots+[0,u].$ If $m([0,u])$ is bounded in $\mathbb R$, then
$m([0,u]) \subseteq [\alpha,\beta]$ for some $\alpha,\beta \in
\mathbb R.$  Then $m([0,nu]) \subseteq [n\alpha,n\beta].$
\end{proof}

\begin{proposition}\label{pr:3.2}
Let $E$ be a pseudo effect algebra with {\rm (RDP)}   and let
$m:E\to \mathbb R $ be a signed measure. Then $m$ is relatively
bounded if and only if $m=m_1-m_2$ for some measures $m_1,m_2$ on
$E.$
\end{proposition}

\begin{proof}
Assume that $m = m_1-m_2$ for some two measures $m_1, m_2 \in
\mathcal M(E)^+.$ If $W \subseteq [a,b]$ in $E,$ then
$m_1(W)\subseteq [m_1(a),m_1(b)]$ and $m_2(W)\subseteq
[m_2(a),m_2(b)].$ Then $m_1(a)-m_2(b) \le m_1(b)-m_2(a)$ and
$m(W)\subseteq [m_1(a)-m_2(b), m_1(b)-m_2(a)]$ that proves that $m$
is relatively bounded.

Conversely, let $m$ be relatively bounded. If we set $d(x):=
m(x)\vee 0$ for all $x\in E,$ then $d(0)=0.$  For all $x,y \in E$
such that $x+y$ is defined in $E,$ we have
$$
d(x+y) = (f(x)+f(y))\vee 0 \le (f(x)\vee 0)+ (f(y)\vee 0) =
d(x)+d(y),
$$
so that $d$ is subadditive.

Let us define $D(x)$ by (3.1) for each $x \in E$. We assert that
$D(x)$ is bounded above in $\mathbb R.$  By the assumption, there
are elements $a,b \in H$ such that $f([0,x])\subseteq [a,b].$  Fix a
decomposition $x = x_1+\cdots+x_n$ with $x_i\in E$ for each
$i=1,\ldots, n.$ By \cite[Lem 1.21]{Goo}, we have

$$ \sum_{i=1}^nd(x_i) =\sum_{i=1}^n (m(x_i)\vee 0) =
\left(\bigvee_{A \in 2^n}\left(\sum_{i\in A} m(x_i)\right)\right)
\vee 0.
$$
For all $A \in 2^n,$ we have
$$ 0 \le \sum_{i\in A}x_i \le x,
\mbox{and}\ \sum_{i\in A} m(x_i) = m(\sum_{i\in A}x_i) \le b.
$$
Hence, $d(x_1) + \cdots+d(x_n) \le b\vee 0,$ and consequently,
$b\vee 0$ is an upper bound for $D(x)$ that proves the assertion.

By Proposition \ref{pr:3.1}, there exists a signed measure $m_1$ on
$E$ such that $m_1(x) = \bigvee D(x)$ for all $x \in E.$ Since
$m_1(x)\ge d(x) \ge 0,$ $m_1(x)$ is a measure, and $m_1(x)\ge d(x)
\ge m(x)$ for all $x \in E.$ Hence, $ m_2 = m_1-m$ is a measure on
$E$, too.
\end{proof}

A signed measure $m$ on a pseudo effect algebra $E$ is said to be
{\it Jordan} if $m$ can be expressed as a difference of two positive
measures on $E,$ and let $\mathcal J(E)$ be the set of all Jordan
measures on $E.$  It is clear that $\mathcal J(E)$ is nonempty
because the zero mapping on $E$ belongs to $\mathcal J(E).$

For example, if $1<\dim H < \infty,$ then on $\mathcal L(H)$ there
is a signed measure that is not Jordan, see e.g. \cite[3.2.4]{Dvu},
whilst if $\dim = \aleph_0,$ then by the Dorofeev-Sherstnev Theorem,
every $\sigma$-additive signed measure on $\mathcal L(H)$ is Jordan,
\cite[Thm 3.2.20]{Dvu}.

Proposition \ref{pr:3.2} says that a signed measure $m$ on a pseudo
effect algebra $E$ with (RDP) is Jordan iff $m$ is relatively
bounded.

Given two signed measure $m_1,m_2 \in \mathcal M(E)$, we define
$m_1\le^+ m_2$ whenever $m_2-m_1$ is a positive measure. Then
$\le^+$ is a partial order on $\mathcal M(E)$ and $\mathcal M(E)$ is
an Abelian po-group with respect to this partial order.

Let $(G;+,0,\le)$ be a  po-group. A subgroup $H$ of $G$ is said to
be {\it convex} if from $x\le y \le z,$ where $x,z\in H$ and $y \in
G,$ we have $y \in H.$ An {\it o-ideal} is any directed convex
subgroup of $G.$

\begin{proposition}\label{pr:3.3}
Let $E$ be a pseudo effect algebra with {\rm (RDP)},  let $\mathcal
J(E)$ be the set of all Jordan signed measures  on $E.$  Then
$\mathcal J(E)$ is a nonempty o-ideal of the po-group $\mathcal
M(E).$
\end{proposition}

\begin{proof}
Due to Proposition \ref{pr:3.2}, $\mathcal J(E)$ equals  the
subgroup of $\mathcal M(E)$ generated by the positive measures.
Therefore, $\mathcal J(E) $ is a directed subgroup of $\mathcal
M(E).$

Given $m_1 \in \mathcal M(E)$ and $m_2 \in \mathcal J(E)$ such that
$0\le^+ m_1 \le^+ m_2,$ write $m_2=m_1'-m_2'$ for some  measures
$m_1',m_2' \in \mathcal M(E)^+.$ Since $m_1\le^+ m_2\le ^+  m_1',$
we have $m_1=m_1' - (m_1'-m_1)$ with $m_1'$ and $m_1'-m_1$ positive
measures, and hence, $m_1 \in \mathcal J(E).$ This proves that
$\mathcal J(E)$ is an o-ideal of $\mathcal M(E).$
\end{proof}

\begin{theorem}\label{th:3.4} Let $E$ be a pseudo effect algebra
with {\rm (RDP)}.

\begin{enumerate}

\item[(a)] The group $\mathcal J(E)$ of all Jordan signed measures on $E$ is
an Abelian Dedekind complete $\ell$-group.

\item[(b)] If $\{m_i\}_{i\in I}$ is a nonempty system of $\mathcal J(E)$
that is bounded above, and if $d(x)=\bigvee_i m_i(x)$ for all $x \in
E,$ then
$$ \left(\bigvee_i m_i\right)(x) = \bigvee\{d(x_1)+\cdots + d(x_n):
x= x_1+\cdots + x_n, \ x_1,\ldots, x_n \in E\}
$$
for all $x \in E.$

\item[(c)] If $\{m_i\}_{i\in I}$ is a nonempty system of $\mathcal J(E)$
that is bounded below, and if $e(x)=\bigwedge_i f_i(x)$ for all $x
\in E,$ then
$$ \left(\bigwedge_i m_i\right)(x) = \bigwedge\{e(x_1)+\cdots + e(x_n):
x= x_1+\cdots + x_n, \ x_1,\ldots, x_n \in E\}
$$
for all $x \in E.$

\end{enumerate}
\end{theorem}

\begin{proof} Let $t \in \mathcal J(E)$ be an upper bound for $\{m_i\}.$
For any $x \in E$, we have $m_i(x)\le t(x),$ so that the mapping
$d(x)=\bigvee_i m_i(x)$ defined on $E$ is a a subadditive mapping.
For any $x \in E$ and any decomposition $x = x_1+\cdots + x_n$ with
all $x_i \in E,$ we conclude $d(x_1)+\cdots+ d(x_n)\le t(x_1)+\cdots
+ t(x_n)=t(x).$  Hence, $t(x)$ is an upper set for $D(x)$ defined by
(3.1).

Proposition \ref{pr:3.1} entails there is a signed measure $m$ on
$E$ such that $m(x)=\bigvee D(x).$ For every $x \in E$ and every
$m_i$ we have $m_i(x)\le d(x)\le m(x)$ that gives $m_i \le^+ m.$ The
mappings $m-m_i$ are positive measures belonging bo $\mathcal
M(E)^+$ that gives $m \in \mathcal J(E).$ If $h \in \mathcal J(E)$
such that $m_i\le^+ h$ for each $i \in I,$ then $d(x)\le h(x)$ for
any $x \in E.$ As above, we can show that $h(x)$ is also an upper
bound for $D(x)$, whence $m(x)\le h(x)$ for any $x \in E$ that gives
$m\le ^+h.$ In other words, we have proved that $m$ is the supremum
of $\{m_i\}_{i\in I},$ and its form is given by (b).

Applying the order anti-automorphism $z\mapsto - z$ in $\mathbb R$,
we see that infima exist in $\mathcal J(E)$ for any bounded below
system $\{m_i\}_{i\in I}$, and their form is given by (c).

By Proposition \ref{pr:3.2}, $\mathcal J(E)$ is directed, combining
(b) and (c), we see that $\mathcal J(E)$ is an Abelian Dedekind
complete $\ell$-group.
\end{proof}

For finite joins and meets of Jordan signed measures, Theorem
\ref{th:3.4} can be reformulated as follows.

\begin{theorem}\label{th:3.5} If $E$ is a pseudo effect algebra with
{\rm (RDP)}, then the group $\mathcal J(E)$ of all Jordan signed
measures on $E$ is an Abelian Dedekind complete lattice ordered real
vector space. Given $m_1,\ldots, m_n \in \mathcal J(E)$,

\begin{eqnarray*}
\left( \bigvee_{i=1}^n m_i\right)(x) = \sup\{m_1(x_1)+\cdots
+m_n(x_n):
x = x_1+\cdots +x_n,\ x_1,\ldots, x_n \in E\},\\
\left( \bigwedge_{i=1}^n m_i \right)(x) = \inf\{m_1(x_1)+\cdots
+m_n(x_n):
x = x_1+\cdots +x_n,\ x_1,\ldots, x_n \in E\},\\
\end{eqnarray*}
for all $x \in E.$
\end{theorem}

\begin{proof} Due to Theorem \ref{th:3.4}, $\mathcal J(E)$ is an
Abelian Dedekind complete $\ell$-group. It is evident that it is a
Riesz space, i.e., a lattice ordered real vector space.

Take $m_1,\ldots,m_n \in \mathcal J(E)$ and let $m = m_1\vee \cdots
\vee m_n.$ For any $x \in E$ and $x=x_1+\cdots+x_n$ with
$x_1,\ldots, x_n \in E,$ we have  $m_1(x_1)+\cdots + m_n(x_n) \le
m(x_1)+\cdots + m(x_n) = m(x).$  Due to Theorem \ref{th:3.4}, given
an arbitrary real number $\epsilon >0$, there is a decomposition $x
= y_1+\cdots+y_k$ with $y_1,\ldots,y_k \in E$ such that

$$ \sum_{j=1}^k \max\{m_1(y_j),\ldots,m_n(y_j)\} > m(x)-\epsilon.
$$

If $k < n$, we can add the zero elements to the decomposition, if
necessary, so that without loss of generality, we can assume that
$k\ge n.$

We note that if $a,b \in E$ are given such that $a+b$ is defined in
$E$, the elements $a',a''\in E$ such that $a+b= b+a'$ and $b+a
=a''+b$ are said to be (right and left) conjugates of $a$ by $b$.
Since $\mathbb R$ is Abelian, for any $h\in \mathcal J(E),$
$h(a')=h(a)=h(a'').$

We decompose the set $\{1,\ldots,k\}$ into mutually disjoint sets
$J(1),\ldots,J(n)$ such that
$$J(i):=\{j \in \{1,\ldots,k\}: \max\{m_1(y_j),\ldots, m_n(y_j)\}=
m_i(y_j)\}.
$$
Assume $J(1)=\{j_{t_1},\ldots, j_{n_1}\}.$  Then the element
$x_1:=x_{j_{t_1}}+\cdots+ x_{j_{n_1}}$ is defined in $E$.

The element $x$ can be expressed in the form $x=x_{j_{t_1}}+\cdots+
x_{j_{n_1}} + x_j' +\cdots + x_k',$ where $x_j',\ldots, x_k'\in E$
are conjugates of $x_j,\ldots,x_k.$

In a similar way, let $J(2)=\{j_{t_2},\ldots, j_{n_2}\}$ and let
$x_2 = y_{j_{t_2}}+\cdots +y_{j_{n_2}}.$ Again, we can express $x$
in the form $x = x_1 + x_2 + y_s'' +\cdots + y_k'',$ where $y_t''$'s
are appropriate  conjugates of $y_s',\ldots, y_k'.$ Processing in
this way for each $J(i) = \{j_{t_i},\ldots, j_{n_i}\},$ we define
the element $x_i = c_{t_{j_{t_i}}} +\cdots + c_{t_{j_{n_i}}},$ where
$c_{t_{j_s}}$ is an appropriate conjugate of the element
$y_{t_{j_s}}.$  Then $x = x_1+\cdots + x_n,$ and

$$ \sum_{i=1}^n m_i(x_i) = \sum_{i=1}^n \sum_{j \in
J(i)}m_i(y_j)=\sum_{i=1}^k\max\{m_1(y_j),\ldots, m_n(y_j)\} >
m(x)-\epsilon.$$

This implies $m(x)$ equals the given supremum.

The formula for $(m_1\wedge \cdots \wedge m_n)(x)$ can be obtained
applying   the order anti-automorphism $z\mapsto -z$ holding in
$\mathbb R.$
\end{proof}

\section{Jordan Signed Measures}

Using the results of the previous Section, we will show some
interesting properties of signed measures, like a Jordan
decomposition, variation, etc.

Let $E$ be a pseudo effect algebra with (RDP), and let $0:E\to
\{0\}$ be the zero signed measure.  Then $\mathcal J(E)$ is a
nontrivial Abelian $\ell$-group, i.e., $\mathcal J(E) \supset \{0\}$
iff $E$ admits at least one state.  Moreover, $0$ is the zero
element of the $\ell$-group $\mathcal J(E).$  We recall that if $E$
is an effect algebra with (RDP), then $\mathcal S(E)$ is always
nonempty.

We say that a convex subset $F$ of a convex set $K$ is a {\it face}
if $x=\lambda x_1+(1-\lambda)x_2 \in F,$ $0<\lambda<1,$ entail
$x_1,x_2\in F.$ For example, if $x$ is an extreme point of $K,$ then
the singleton $\{x\}$ is a face, and for any $X \subseteq K,$ there
is the face generated by $X.$ Due to \cite[Prop 5.7]{Goo}, the face
$F$ generated by $X$ is the set of those points $x \in K$ for which
there exists a positive convex combination $\lambda x + (1-\lambda)
y= z$ with $y \in K$ and $z$ belonging to the convex hull of $X.$

In particular, the face of $K$ generated by a point $z\in K$
consists precisely of those points $x \in K$ for which there exists
a positive convex combination $\lambda x +\beta y =z$ with $y \in
K.$

\begin{lemma}\label{le:ex3.7}  Let $E$ be a pseudo effect algebra
and let $X$ be a subset of $\mathcal S(E).$ Then a state $s\in
\mathcal S(E)$ belongs to the face generated by $X$ if and only if
$s\le ^+ \alpha t$ for some positive constant $\alpha$ and some
state $t$ in the convex hull of $X.$
\end{lemma}

\begin{proof}  If a state $s$ belongs to the face generated by $X,$
by the note just before this lemma, there exists a positive convex
combination $\lambda s+(1-\lambda)s' = t,$ where $s'\in \mathcal
S(E)$ and $t$ belongs to the convex hull of $X.$ Then $\lambda s
\le^+ t$ so that $s\le ^+ t/\lambda.$

Conversely, if $s \le^+ \alpha t$ for some $\alpha >0$ and some
state $t$ in the convex hull of $X.$  Then $\alpha t - s$ is a
measure, so that $\alpha t - s = \beta s'$ for some $\beta \ge 0.$
Now $s+\beta s'= \alpha t$ and $1+\beta = s(1) + \beta s'(1) =
\alpha t(1) =\alpha$ that yields $1/\alpha + \beta/\alpha = 1.$ This
gives $0\le \lambda:= 1/\alpha \le 1$ and $\lambda s +
(1-\lambda)s'=t.$ Since $t$ belongs to the face generated by $X,$ so
does $s.$
\end{proof}

Now we show that if $s_1$ and $s_2$ are two  states on $E,$ then
$s_1\wedge s_2$ and $s_1\vee s_2$ are not necessarily states.

\begin{proposition}\label{pr:ex3.8}  Let $E$ be a pseudo effect
algebra with {\rm (RDP)}.  Let $F_1$ and $F_2$ be the faces
generated by states $s_1$ and $s_2,$ respectively, on $E.$ The
following statements are equivalent:

\begin{enumerate}

\item[(i)] $F_1 \cap F_2 = \emptyset.$ 

\item[(ii)] $s_1\wedge s_2 = 0.$

\item[(iii)] $s_1\vee s_2 = s_1 + s_2.$

\item[(iv)] Given $x \in E$ and any $\epsilon >0,$ there exists
$x_1,x_2 \in E$ such that $x = x_1+x_2$ and $s(x\minusli
x_i)<\epsilon.$

\end{enumerate}

In particular, if $s_1$ and $s_2$ are two distinct extremal state on
$E,$ then $s_1\wedge s_2 =0.$
\end{proposition}

\begin{proof}
(i) $\Rightarrow$ (ii). Assume that $s_1$ and $s_2$ belongs to
mutually disjoint faces of $\mathcal S(E).$ If $s_1\wedge s_2 >0,$
there is a state $s$ and a real number $\alpha >0$ such that
$s_1\wedge s_2 = \alpha s.$ Then $\alpha s\le^+ s_1$ and $\alpha
s\le ^+ s_2$ and $s\le^+ s_1/\alpha$ and $s\le^+s_2/\alpha.$ Lemma
\ref{le:ex3.7} implies $s$ belongs to the face generated by $s_1$
and as well to the one generated by $s_1$ that is absurd, so that
$s_1 \wedge s_2=0.$

(ii) $\Rightarrow$ (i). Let $F_1$ and $F_2$ be the faces generated
by $s_1$ and $s_2,$ respectively. We state that $F_1 \cap F_2 =
\emptyset.$ If not, there is a state $s \in F_1 \cap F_2$ and by
Lemma \ref{le:ex3.7}, $s\le ^+ \alpha_1 s_1$ and $s\le ^+\alpha_2
s_2.$ If $\alpha = \max\{\alpha_1,\alpha_2\}$, then $s\le^+ \alpha
s_1$ and $s\le^+ \alpha s_2$ and $s/\alpha \le^+ s_1, s_2$ and
therefore, $s/\alpha \le^+ s_1 \wedge s_2 = 0$ that gives a
contradiction.

(i) $\Leftrightarrow$ (iii). It follows from the basic properties of
$\ell$-groups, see e.g. \cite[C p.67]{Fuc},  $(s_1\vee s_2)+ (s_1
\wedge s_2) = s_1 + s_2.$

(ii) $\Rightarrow$ (iv).  By Theorem \ref{th:3.5}, given $x$ and
$\epsilon >0,$ there are $x_1',x_2'\in E$ such that $x=x_2'+x_1'$
and $s_1(x_2')+s_2(x_1') < \epsilon.$  Then $x_2'= x\minusli x_1'$
and $x_1'= x_2'\minusre x,$ and $s_1(x_2')= s_1(x\minusli x_1')
<\epsilon$ and $s_2(x_1')= s_1(x_2'\minusre x) = -s_1(x_2')+s_1(x)=
s_1(x)-s_1(x_2') = s_1(x\minusli x_2')<\epsilon.$

But $x= x_1'+ x_2'',$ where $x_2''$ is a conjugate of $x_2'$ by
$x_1'.$ Then $s_2(x_2'')= s_2(x_2').$ If we set $x_1 = x_1'$ and
$x_2= x_2'',$ we have $s_1(x\minusli x_1)=s_1(x\minusli
x_1')<\epsilon,$ $s_2(x\minusli x_2)= s_2(x)- s_2(x_2)= s_2(x)-
s_2(x_2'') = s_2(x)-s_2(x_2') = s_2(x\minusli x_2')<\epsilon$ and
$x= x_1+x_2.$

(iv) $\Rightarrow$ (ii).  Given $\epsilon >0$ and $x \in E,$ there
is a decomposition $x=x_1+x_2$ such that $s_i(x\minusli x_i)
<\epsilon/2$ for $i=1,2.$ Hence, $x= x_2 + x_1',$ where $x_1'$ is a
conjugate of $x_1.$ If we set $y = x_2$ and $z=x_1',$ then using
$x_1=x\minusli x_2$ and $x_2= x_1\minusre x,$ we have $s_1(y) =
s_1(x_2) = s_1(x_1\minusre x)= s_1(x\minusli x_1)<\epsilon/2$ and
$s_2(z)= s_2(x_1') = s_2(x\minusli x_2)<\epsilon /2$, so that
$s(y)+s(z)< \epsilon.$ By Theorem \ref{th:3.5}, this means
$s_1\wedge s_2 =0.$

Finally, if $s_1$ and $s_2$ are two distinct extremal states, then
the singletons $\{s_1\}$ and $\{s_2\}$ are mutually disjoint faces.
Hence, $s_1 \wedge s_2 = 0.$
\end{proof}

\begin{proposition}\label{co:ex3.9}
Let $s_1,s_2$ be two states on a pseudo effect algebra $E$ with {\rm
(RDP)}.  Then $s_1 \wedge s_2 \in  \mathcal S(E)$ if and only if
$s_1=s_2$ and if and only if $s_1\vee s_2$ is a state.

Given $\lambda \in [0,1],$ let $s_\lambda:=\lambda s_1
+(1-\lambda)s_2 \in \mathcal S(E).$  Then $s_1\wedge s_2 =
\bigwedge\{s_\lambda: \lambda \in [0,1]\} \in \mathcal M^+(E).$

\end{proposition}

\begin{proof}
Let $s = s_1 \wedge s_2 \in \mathcal S(E).$ Then $s\le^+s_1$ and
$s\le^+ s_2$. Therefore, $s_i-s$ is a positive measure. Since
$s_i(1)-s(1)=0$ for $i=1,2,$ we see that $s_1=s=s_2.$ The converse
statement is evident. The second equivalency follows from the
$\ell$-group equality $(s_1\wedge s_2)+(s_1\vee s_2)=s_1 +s_2.$

Let $s = s_1\wedge s_2 \in \mathcal M^+(E).$ Given $\lambda \in
[0,1],$ we have $\lambda s\le^+ \lambda s_1$ and $(1-\lambda)s \le^+
(1-\lambda)s_2$ so that $s= \lambda s +(1-\lambda)s \le ^+ \lambda
s_1 + (1-\lambda)s_2.$ Hence $s\le^+s_0:=\bigwedge \{s_\lambda:
\lambda \in [0,1]\}.$ If we set $\lambda = 1$ or $\lambda =0,$ we
see that $s_1,s_2 \in \{s_\lambda: \lambda \in [0,1]\}.$ Therefore,
$s_0 \le^+ s.$
\end{proof}

A signed measure $m$ on a pseudo effect algebra $E$ is $\sigma$-{\it
additive} if, $\{a_n\} \nearrow a,$ i.e. $a_n\le a_{n+1}$ for each
$n\ge 1$ and $\bigvee_n a_n = a,$ then $m(a)=\lim_n m(a_n).$ A
measure $m$ is $\sigma$-additive iff $a_n\searrow 0$ entails
$m(a_n)\searrow 0.$

\begin{proposition}\label{pr:3.8} If $m_1$ and $m_2$ are
$\sigma$-additive measures on a pseudo effect algebra with {\rm
(RDP)}, so are $m_1\vee m_2$ and $m_1\wedge m_2.$
\end{proposition}

\begin{proof} Let $a_n\searrow 0.$ Due to Theorem \ref{th:3.5},
$m_1(a_n)+ m_2(a_n) \ge (m_1\vee m_2)(a_n) \ge 0$ so that $(m_1\vee
m_2)(a_n) \searrow 0.$ Similarly, $m_1(a_n) \ge (m_1 \wedge
m_2)(a_n) \ge 0$ and $(m_1\wedge m_2)(a_n)\searrow 0.$
\end{proof}

Theorem \ref{th:3.5} allows us to define, for any Jordan signed
measure $m,$ its positive and negative parts, $m^+$ and $m^-,$ via

$$ m^+:= m\vee 0 \quad \mbox{and}\quad  m^-:= -(m\wedge 0).$$
Then $m = m^+ - m^-,$ $(-m)^+= m^-,$ and $(-m)^-= m^+.$

Theorem \ref{th:3.5} says that
$$ m^+(a) = \sup\{m(x): x\le a\} \quad \mbox{and} \quad m^-(a)=\inf\{m(x): x
\le a\}\eqno(4.1)
$$
for each $a \in E.$

The decomposition $m=m^+-m^-$ is said to be  {\it Jordan}, and if
$m= m_1-m_2$ for some positive measures $m_1,m_2 $ on $E$, then
$m^+\le^+ m_1$ and $m^-\le^+ m_2.$ Moreover, we define an {\it
absolute value}, $|m|,$ of $m$ via
$$ |m|=m^++m^-.$$

Therefore, if $\mathcal S(E)\ne \emptyset,$ every  Jordan signed
measure $m$ can be uniquely expressed in the form
$$ m = \alpha_1 s_1 -\alpha_2 s_2, \eqno(4.2)
$$
where $\alpha_1,\alpha_2$ are real numbers and $s_1,s_1$ are states
such that $\alpha_1 s_1=m^+$ and $\alpha_2 s_2=m^-,$ we call it a
{\it canonical Jordan decomposition} of $m.$

The measures $m^+,m^-$ and $|m|$ are sometimes called also an {\it
upper} or  {\it positive variation}, a {\it lower} or {\it negative
variation} and a {\it total variation} of $m,$ respectively.

\begin{proposition}\label{pr:3.9}  For any Jordan signed measure $m$
on a pseudo effect algebra $E$ with {\rm (RDP)}, we define a mapping
$v_m: E \to \mathbb R$ by
$$v_m(x):= \sup\{|m(x_1)|+\cdots + |m(x_n)|: x= x_1+\cdots+x_n, n\ge
1 \}.\eqno(4.3)
$$

Then $v_m = |m|.$

\end{proposition}

\begin{proof}  Let $x = x_1+\cdots+x_n.$  Then $|m(x_1)|+\cdots +
|m(x_n)|\le |m|(x_1)+\cdots + |m|(x_n) = |m|(x),$ so that $v_m(x)
\le |m|(x).$  Due to (4.1), $m^+(x),m^-(x)\le v_m(x)$ for each $x\in
E.$  We assert that $v_m$ is subadditive, i.e., $v_m(x+y) \le
v_m(x)+v_m(y)$ whenever $x+y \in E.$ Indeed, if $x +y = z_1+\cdots
+z_n$, (RDP) entails that there are $x_1,\ldots,z_n,y_1,\ldots, y_n
\in E$ such that $x=x_1+\cdots+z_n$ and $y=z_1+\cdots+ z_n.$  Then
$|m(z_1)|+\cdots+|m(z_n)| \le \sum_i |m(x_i)| + \sum_i |m(y_i)| \le
v_m(x)+v_m(y),$ so that $v_m(x+y)\le v_m(x)+v_m(y).$

According to (3.1), we define the set
$$V_m(x)=\{v_m(x_1)+\cdots + v_m(x_n): x = x_1+\cdots + x_n\}.
$$
This set is bounded in $\mathbb R$, its upper bound is $|m(x)|.$
Proposition \ref{pr:3.1} yields that the functional $V(x)=\sup
V_m(x),$ $x \in E,$ is a positive measure on $E.$ It is clear that
$v_m(x)\le V(x)\le |m|(x)$ for each $x \in E.$

We show that $v_m = V.$  Given $\epsilon >0,$ there is a
decomposition $x = x_1+\cdots+x_n$ such that $\sum_i v_m(x_i)>
V(x)-\epsilon.$  For any $i=1,\ldots,n$, there is a finite
decomposition of each $x_i = \sum_j x_i^j$ such that
$\sum_j|m(x_i^j)| \ge v_m(x_i)- \epsilon/n.$ Therefore,
$$ \sum_{i=1}^n \sum_j |m(x_i^j)| > \sum_{i=1}^n (v_m(x_i)- \epsilon /n) = \sum_{i=1}^n
v_m(x_i) - \epsilon > V(x)-2\epsilon.
$$
This entails, $v_m(x) \ge V(x)-2\epsilon.$  Since $\epsilon$ was
arbitrary, $v_m(x)\ge V(x),$ consequently, $v_m(x)=V(x)$ for any
$x\in E.$

Since $m^+\wedge m^-=0,$ the $\ell$-group properties imply $|m| =
m^+ + m^- = m^+ \vee m^-.$  Since $m^+, m^+\le^+ v_m=V \le^+ |m|,$
we have $v_m = |m|.$
\end{proof}

Let $\{a_n\}$ be a sequence of elements of a pseudo effect algebra
$E$ such that $b_n=a_1+\cdots+a_n$ exists for each $n\ge 1.$  If $a
= \bigvee_n b_n$ is defined in $E,$ we write $a := a_1+a_2 +\cdots
=\sum_n a_n.$ Let $\{a_n\}\nearrow a.$ If we set $a_1'=a_1$ and
$a_n' = a_{n-1}\minusre a_n$ for each $n\ge 2,$ then $a_1'+\cdots +
a_n'=a_n$ for each $n\ge 1$ and $a=\sum_n a_n'.$  Hence, a signed
measure $m$ is $\sigma$-additive iff $a=\sum_n a_n$ entails
$m(a)=\sum_n m(a_n).$

A pseudo effect algebra is said to be  {\it monotone
$\sigma$-complete} if $a_1\le a_2 \le a_n \le \cdots,$ then $a
=\bigvee_n a_n$ is defined in $E.$ We say that $E$ satisfies
$\sigma$-(RDP) if $a_1+a_2 = b_1+b_2 + \cdots ,$ then there are two
sequences $\{c_{1n}\}_n$ and $\{c_{2n}\}$ such that $a_i = \sum_n
c_{in}$ for $i=1,2$ and $b_n= c_{1n}+c_{2n}$ for each $n\ge 1.$

\begin{proposition}\label{pr:3.10}  Let $E$ be a pseudo effect
algebra with $\sigma$-{\rm (RDP)}. If $m$ is a $\sigma$-additive
Jordan signed measure, so is $m^+, m^-$ and $|m|.$
\end{proposition}

\begin{proof}  Assume $\{a_n\}\nearrow a.$ If we set $a_1'=a_1$ and
$a_n' = a_{n-1}\minusre a_n$ for each $n\ge 2,$ then $a_1'+\cdots +
a_n'=a_n$ for each $n\ge 1$ and $a=\sum_n a_n'.$ Then $m(a) = \sum_n
m(a_n').$

We show that $|m|=v_m$ is $\sigma$-additive. We have $v_m(a) \ge
v_m(a_n)$ so that $v_m(a) \ge \lim v_m(a_n).$  Now assume $a=
x_1+\cdots+x_k.$  The $\sigma$-(RDP) entails that there is $k$ many
sequences $\{c_{jn}\}_n$ for $j=1,\ldots, k$ such that $x_j = \sum_n
c_{jn}$ and $a_n'= \sum_{j=1}^k c_{jn}$ for each $n\ge 1.$ Check

\begin{eqnarray*}
\sum_{j=1}^k |m(x_j)|&=& \sum_{j=1}^k |m(\sum_n c_{jn})|=
\sum_{j=1}^k|\sum_n m(c_{jn})|\\
&\le& \sum_n \sum_{j=1}^k |m(c_{jn})| \le \sum_n v_m(a_n),
\end{eqnarray*}
so that $v_m(a) \le \sum_n \le \lim_n v_m(a_n)$ and $v_m=|m|$ is
$\sigma$-additive.  Because $m +2m^+ = |m|,$ we see that $m^+$ is
$\sigma$-additive, consequently, so is $m^-.$
\end{proof}

Suppose that $E$ admits at least one state.  Given a positive
measure $m$ on $E$ with $m(1)>0,$ let $\mathcal J(m)=[0,m]:=\{t \in
\mathcal J(E): 0\le^+ t \le ^+m\}$ be an interval in $\mathcal
J(E).$  We can define on it an MV-structure by $s\oplus t:=
(s+t)\wedge m,$ $s\odot t:= \{s+t-m\}\vee 0,$ and $s^*= m- s$ for
all $s,t \in \mathcal J(m).$ Then $(\mathcal J(m); \oplus,^*, 0)$ is
an MV-algebra, where $m=0^*$ is the unit element of $\mathcal J(m).$

The partial operation, $+$, on $\mathcal J(m),$ is defined as
follows: $s+t$ is defined in $\mathcal J(m)$ iff $s+t \le^+ m,$ or
equivalently, it coincides with the restriction of the standard
addition of the functions $s$ and $t$ belongs to $\mathcal J(m).$

It is clear that the state space of $\mathcal J(m)$ is non-void. Let
$a\in E$ be a fixed element. The mapping $\mu_a: \mathcal J(m)\to
[0,1]$ defined by $\mu_a(s):= s(a),$ $s \in \mathcal J(m),$ is a
state on $\mathcal J(E).$

Moreover, the system of states $\{\mu_a: a\in E\}$ is {\it
order-determining}, i.e. $\mu_a(s)\le \mu_a(t)$ for all $a\in E,$
implies $s \le^+ t.$

\section{Simplex Structure of Pseudo Effect Algebras and Integrals}

This is the main section of the paper. We show that if a pseudo
effect algebra satisfies (RDP), then its state space is either empty
or a non-empty simplex. This will allow represent states by standard
integrals.

The following notions on convex sets can be found e.g. in
\cite{Goo}. Let $K_1, K_2$ be two convex sets. A mapping $f:K_1\to
K_2$ is said to be {\it affine} if it preserves all convex
combinations, and if $f$ is also injective and surjective such that
also $f^{-1}$ is affine, $f$ is said to be an {\it affine
isomorphism} and $K_1$ and $K_2$ are {\it affinely isomorphic}.

We recall that a {\it convex cone} in a real linear space $V$ is any
subset $C$ of  $V$ such that (i) $0\in C,$ (ii) if $x_1,x_2 \in C,$
then $\alpha_1x_1 +\alpha_2 x_2 \in C$ for any $\alpha_1,\alpha_2
\in \mathbb R^+.$  A {\it strict cone} is any convex cone $C$ such
that $C\cap -C =\{0\},$ where $-C=\{-x:\ x \in C\}.$ A {\it base}
for a convex cone $C$ is any convex subset $K$ of $C$ such that
every non-zero element $y \in C$ may be uniquely expressed in the
form $y = \alpha x$ for some $\alpha \in \mathbb R^+$ and some $x
\in K.$

We recall that in view of \cite[Prop 10.2]{Goo}, if $K$ is a
non-void convex subset of $V,$ and if we set

$$ C =\{\alpha x:\ \alpha \in \mathbb R^+,\ x \in K\},
$$
then $C$ is a convex cone in $V,$ and $K$ is a base for $C$ iff
there is a linear functional $f$ on $V$ such that $f(K) = 1$ iff $K$
is contained in a hyperplane in $V$ which misses the origin.

Any strict cone $C$ of $V$ defines a partial order $\le_C$ via $x
\le_C y$ iff $y-x \in C.$ It is clear that $C=\{x \in V:\ 0 \le_C
x\}.$ A {\it lattice cone} is any strict convex cone $C$ in $V$ such
that $C$ is a lattice under $\le_C.$

A {\it simplex} in a linear space $V$ is any convex subset $K$ of
$V$ that is affinely isomorphic to a base for a lattice cone in some
real linear space. A  simplex $K$ in a locally convex Hausdorff
space is said to be (i) {\it Choquet} if $K$ is compact, and (ii)
{\it Bauer} if $K$ and $\partial_e K$ are compact, where $\partial_e
K$ is the set of extreme points of $K.$

A simplex is a generalization of a classical simplex in $\mathbb
R^n,$ and we recall that no disc or no convex quadrilateral in the
plane are not simplices.

\begin{theorem}\label{th:4.2}  If $E$ is a pseudo effect algebra with
{\rm (RDP)}, then either $\mathcal S(E)$ is empty or it is a
nonempty Choquet simplex.
\end{theorem}

\begin{proof} Assume that $\mathcal S(E)$ is nonempty. Then the
positive cone  $\mathcal M(E)^+= \mathcal J(E)^+$ of the Abelian
Dedekind complete $\ell$-group $\mathcal J(E)$ consists of all
positive measures on $E,$ so that $\mathcal J(E)^+ =\{\alpha s:
\alpha \in \mathbb R^+, s \in \mathcal S(E)\}.$ Since  $\mathcal
S(E)$ lies in the hyperplane $\{m \in \mathcal J(E): m(u)= 1\}$
which misses the origin,   $\mathcal S(E)$ is a base for $\mathcal
J(E)^+,$ and $\mathcal S(E)$ is a simplex.  On the other hand,
$\mathcal S(E)$ is compact, so that $\mathcal S(E)$ is a Choquet
simplex.
\end{proof}

We note that if $E$  is an effect algebra with (RDP) and $0\ne 1,$
then $E$ admits at least one state because then $E=\Gamma(G,u)$ for
some  unital Abelian interpolation po-group $(G,u)$; now it is
enough to apply \cite[Cor 4.4]{Goo}. Hence, its state space is
always a non-empty Choquet simplex. If an effect algebra $E$ does
not satisfy (RDP), then its state space is not necessarily a
simplex; for instance, this is the case  for $E=\mathcal E(H),$
$\dim H>2.$ On the other hand, the state space of a commutative
C$^*$-algebra or the trace space of a general C$^*$ are simplices,
\cite[Thm 4.4, p. 7]{AlSc}  or \cite[Ex 4.2.6]{BrRo}.

On the other hand, it is important to recall that according to a
delicate result of Choquet \cite[Thm I.5.13]{Alf}, for any pseudo
effect algebra $E,$ $\partial_e {\mathcal S}(E)$ is always a Baire
space in the relativized topology induced by the topology of
${\mathcal S}(E)$, i.e. the Baire Category Theorem holds for
$\partial_e {\mathcal S}(E).$

\begin{remark}\label{re:4.2} Theorem {\rm \ref{th:4.2}} was proved for a
pseudo effect algebra that is  an interval pseudo effect algebra,
i.e., $E = \Gamma(G,u),$ for a unital po-group $(G,u)$ with {\rm
(RDP)}. However, we do not know whether every pseudo effect algebra
with {\rm (RDP)} is an interval pseudo effect $\Gamma(G,u)$, where
also $(G,u)$ satisfies {\rm (RDP)}, it was necessary to prove
Theorem {\rm \ref{th:4.2}} in full details.
\end{remark}

If a pseudo effect algebra $E$ satisfies (RDP)$_2,$ then according
to \cite[Thm 4.4]{Dvu3}, the state space of $E$ is either the empty
set or a nonempty Bauer simplex.

\begin{example}\label{ex:3.7}
There is a pseudo effect algebra $E$  with {\rm (RDP)} but {\rm
(RDP)$_2$} fails to hold in $E$ such that $\mathcal S(E)$ is a
non-void Bauer simplex.
\end{example}

\begin{proof}
Let $\mathbb Q$ be the set of all rational numbers and let
$G=\mathbb Q\times \mathbb Q$ be ordered by the strict ordering,
i.e. $(g_1,g_2)\le (h_1,h_2)$ iff $g_1 < h_1$ and $g_2< h_2$ or
$g_1=h_1$ and $g_2=h_2.$ If we set $u=(1,1),$ then $E=\Gamma(G,u)$
is an effect algebra with (RDP)  that is not a lattice. If
$s_0(g,h):=h$ and $s_1(g,h):=g,$ then $s_0$ and $s_1$ are unique
extremal states on $E,$ and every state $s$ is of the form $s =
s_\lambda:=\lambda s_1 +(1-\lambda)s_0,$ $\lambda \in [0,1],$ for
more details, see \cite[Ex 4.2]{BCD}.
\end{proof}

A pseudo effect algebra $E$ has the {\it Bauer simplex property}
((BSP) for short), if $\mathcal S(E)$ is a non-void Bauer simplex.

Let $K$ be a compact convex subset of a locally convex Hausdorff
space. A mapping $f:\ K \to \mathbb R$ is said to be {\it affine}
if, for all $x,y \in K$ and any $\lambda \in [0,1]$, we have
$f(\lambda x +(1-\lambda )y) = \lambda f(x) +(1-\lambda ) f(y)$. Let
$\Aff(K)$ be the set of all continuous affine functions on $K.$ Then
$\mbox{Aff}(K)$ is a unital po-group with the strong unit $1$ which
is a subgroup  of the po-group $\mbox{C}(K)$ of all continuous
real-valued functions on $K$ (we recall that, for $f,g \in
\mbox{C}(K),$ $f \le g$ iff $f(x)\le g(x)$ for any $x \in K$). In
addition, $\mbox{C}(K)$ is  an $\ell$-group and the function $1$ is
its strong unit.

We note that   if $E$ is a pseudo effect algebra such that $\mathcal
S(E)\ne \emptyset,$  given $a \in E,$ let $\hat a:\mathcal S(E) \to
[0,1]$ such that $\hat a(s):= s(a),$ $s \in \mathcal S(E).$ Then
$\hat a \in \Aff(\mathcal S(E)).$

If  $K$ is a compact Hausdorff topological space, let ${\mathcal
B}(K)$ be the Borel $\sigma$-algebra of $K$ generated by all open
subsets of $K.$  Let  ${\mathcal M}_1^+(K)$ denote the set of  all
probability measures, that is, all positive regular
$\sigma$-additive Borel measures $\mu$ on $\mathcal B(K).$  We
recall that a Borel measure $\mu$ is called {\it regular} if each
value $\mu(Y)$ can be approximated by closed subspaces of $Y$ as
well by open subsets $O$ such that $Y \subseteq O.$

We recall that if $x \in K,$ then the Dirac measure $\delta_x$
defined by $\delta(A) :=\chi_A(x),$ $A \in \mathcal B(K),$ is a
regular Borel probability measure.

For two measures $\mu$ and $\nu$ we define the Choquet equivalence
$\sim$ defined by

$$\mu \sim \lambda\quad  \mbox{iff}\quad
\int_K f \dx \mu=\int_K f \dx \lambda,\ f \in \mbox{Aff}(K).
$$

If  $\mu $ and $\lambda$ are nonnegative regular Borel measures on a
convex compact set $K,$  we introduce for them the {\it Choquet
ordering} defined by
$$
\mu \prec \lambda  \quad  \mbox{iff} \quad \int_K f \dx \mu\le\int_K
f \dx \lambda, \ f \in \mbox{Con}(K),
$$  where $\mbox{Con}(K)$ is the set of all continuous convex
functions $f$ on $K$ (that is $f(\alpha x_1+(1-\alpha) x_2)\le
\alpha f(x_1)+(1-\alpha)f(x_2)$ for $x_1,x_2\in K$ and $\alpha \in
[0,1]$). Then $\prec$ is a partial order on the cone of nonnegative
measures. The fact $\lambda \prec \mu$ and $\mu \prec \lambda$
implies $\lambda = \mu$ follows from the fact that
$\mbox{Con}(K)-\mbox{Con}(K)$ is dense in $\mbox{\rm C}(K).$

Moreover, for any probability measure  $\lambda$ there is a maximal
probability measure $\mu$ in Choquet's ordering such that $\mu \succ
\lambda,$ \cite[Lem 4.1]{Phe}.

We recall that the Choquet ordering $\mu \prec \nu$ between two
probability measures $\mu$ and $\nu$ roughly speaking means that
$\nu$ is located further out than $\mu$ towards the set of extremal
points where the convex function have large values, \cite[p.
8]{AlSc}.

\begin{theorem}\label{th:7.2'}
Let $E$ be a pseudo effect algebra with {\rm (RDP)} and with
$\mathcal S(E)\ne \emptyset.$ Let $s$ be a state on $E.$ Let $\psi:
E \to \Aff(\mathcal S(E))$ be defined by $\psi(a) := \hat a,$ $a\in
E,$ where $\hat a$ is a mapping from $\mathcal S(E)$ into $[0,1]$
such that $\hat a(s):=s(a),$ $s \in \mathcal S(E).$ Then there is a
unique state $\tilde s$ on the unital po-group $(\Aff(\mathcal
S(E)),1)$ such that $\tilde s(\hat a) = s(a)$ for any $a \in E.$

The mapping $s \mapsto \tilde s$ defines an affine homeomorphism
from the state space $\mathcal S(E)$ onto $\mathcal
S(\Gamma(\Aff(\mathcal S(E)),1)).$
\end{theorem}

\begin{proof}
Since $E$ is a pseudo effect algebra such that $\mathcal S(E)$ is
non-void, Theorem \ref{th:4.2} asserts that $\mathcal S(E)$ is a
nonempty Choquet simplex.  We define $\widehat E:=\{\hat a: a\in E\}
\subset \Aff(\mathcal S(E))$ and let $\Aff(E)$ be the Abelian
subgroup of $\Aff(\mathcal S(E))$ generated by $\widehat E.$ Given
$s \in \mathcal S(E),$ let $\tilde s$ be a mapping defined on the
unital Abelian po-group with (RDP) $(\Aff(\mathcal S(E)),1)$ such
that $\tilde s(f):=f(s),$ $f \in \Aff(\mathcal S(E)).$  Then $\tilde
s$ is a state on $(\Aff(\mathcal S(E)),1).$

By \cite[Thm 7.1]{Goo}, the mapping $s \mapsto \tilde s$ is an
affine homeomorphism between $\mathcal S(E)$ and $\mathcal
S(\Aff(\mathcal S(E),1).$
\end{proof}

\begin{theorem}\label{th:7.3'}  Let $E$ be a pseudo effect algebra
with {\rm (RDP)} having at least one state. Let $s$ be a state on
$E.$ Then there is a unique maximal regular Borel probability
measure $\mu_s \sim \delta_s$ on $\mathcal B(\mathcal S(E))$ such
that

$$ s(a) = \int_{\mathcal S(E)} \hat a(x) \dx \mu_s(x),\quad a \in
E. \eqno(5.1)$$
\end{theorem}

\begin{proof}
Due to Theorem \ref{th:4.2}, $\mathcal S(E)$ is a nonempty Choquet
simplex. By Theorem \ref{th:7.2'}, there is a unique state $\tilde
s$ on $(\Aff(\mathcal S(E)),1)$ such that $\tilde s(\hat a) = s(a),$
$a \in A.$

Applying the Choquet--Meyer Theorem, \cite[Thm p. 66]{Phe},  we have

$$f(s)=\int_{\mathcal S(E)} f(x) \dx\mu_s, \quad f \in \Aff(\mathcal
S(E)).
$$
Since $\hat a \in \Aff(\mathcal S(E))$ for any $a\in E,$ we have the
representation given by (5.1).
\end{proof}

\begin{theorem}\label{th:7.5'}  Let $E$ be a pseudo effect algebra with
{\rm (BSP)} and let $s$ be a state on $E.$ Then there is a unique
regular Borel probability measure, $\mu_s,$ on $\mathcal B(\mathcal
S(E))$ such that $\mu_s(\partial_e \mathcal S(E))=1$ and

$$ s(a) = \int_{\partial_e \mathcal S(E)} \hat a(x) \dx \mu_s(x),\quad a \in
E. \eqno(5.2)$$
\end{theorem}

\begin{proof}  Due to Theorem \ref{th:7.3'}, we have a unique regular
Borel probability measure $\mu_s\sim \delta_s$ such that (5.1)
holds. The characterization of Bauer simplices, \cite[Thm
II.4.1]{Alf}, says that then $\mu_s$ is a unique regular Borel
probability measure $\mu_s$ on $\mathcal B(\mathcal S(E))$ such that
(4.1) holds and $\mu_s(\partial_e \mathcal S(E)) = 1.$ Hence, (5.2)
holds.
\end{proof}

It is worthy  to remark a note concerning formula (5.2) that if
$\mu$ is any regular Borel probability measure, the right-hand side
of formula (5.1) defines a state, say $s_\mu,$ on $E.$ But if
$\mu(\partial_e \mathcal S(E))<1,$ then for $s_\mu$ there is another
regular Borel probability measure $\mu_0$ such that
$\mu_0(\partial_e \mathcal S(E))=1$ and it represents $s_\mu$ via
(5.2).  For example, take $E$ from Example \ref{ex:3.7}. The state
space of $\mathcal S(E)$ is affinely homeomorphic with the real
interval $[0,1].$  Let $\mu_L$ be the Lebesgue measure on $[0,1].$
Formula (5.1) for $\mu_L$ defines a state $s_L$ on $E$ such that if
$a=(g,h)\in E$ and $s_\lambda(g,h)=\lambda g+(1-\lambda)h,$ then
$s_L(a) = \int_0^1 (\lambda g +(1-\lambda)h) \dx \mu_L(\lambda)=
(g+h)/2.$ So that $s_L= (s_0+s_1)/2,$ but $\mu_L(\partial_e \mathcal
S(E))=0.$ In other words,  $s_L$ has two different representations
by regular Borel probability measures via (5.1) ($\mu_L$ and
$(\delta_0 + \delta_1)/2$, only second one is described by Theorem
\ref{th:7.3'}) and uniquely  via (5.2).

\begin{corollary}\label{co:7.3''}  Let $E$ be a pseudo effect algebra
with {\rm (RDP)} having at least one state. Let $m$ be a Jordan
signed measure on  $E$ and let $m=\alpha s_1 -\beta s_2$ be its
canonical Jordan decomposition. Then there  are unique maximal
regular Borel probability measures $\mu_{s_1} \sim \delta_{s_1}$ and
$\mu_{s_2} \sim\delta_{s_2}$ on $\mathcal B(\mathcal S(E))$ such
that for $\mu_m:= \alpha_1 \mu_{s_1} -\alpha_2 \mu_{s_2}$ we have

$$ m(a) = \alpha_1 \int_{\mathcal S(E)} \hat a(x) \dx \mu_{s_1}(x)
- \alpha_2 \int_{\mathcal S(E)} \hat a(x) \dx \mu_{s_2}(x) =
\int_{\mathcal S(E)} \hat a(x) \dx \mu_{m}(x)
$$
for each $a \in E.$
\end{corollary}

\begin{proof}
Since $m(a) = \alpha_1 s_1(a) - \alpha_2 s_2(a),$ the statement
follows from (4.2) and Theorem \ref{th:7.3'}.
\end{proof}

\begin{theorem}\label{th:7.5''}  Let $E$ be a pseudo effect algebra with
{\rm (BSP)} such that $E$ has at least one state. Let $m$ be a
Jordan signed measure on $E$ and and let $m=\alpha_1 s_1 -\alpha_2
s_2$ be its canonical Jordan decomposition.

Then there  are unique  regular Borel probability measures
$\mu_{s_1},\mu_{s_2} $ on $\mathcal B(\mathcal S(E))$ such that
$\mu_{s_i}(\partial_e \mathcal S(E))=1$ for $i=1,2$ and for $\mu_m:=
\alpha_1 \mu_{s_1} -\alpha_2 \mu_{s_2}$ we have

$$ m(a) = \alpha_1\int_{\partial_e \mathcal S(E)} \hat a(x) \dx \mu_{s_1}(x)-
\alpha_2 \int_{\partial_e \mathcal S(E)} \hat a(x) \dx \mu_{s_2}(x)=
\int_{\partial_e \mathcal S(E)} \hat a(x) \dx \mu_{m}(x)
 $$
for each $a \in E.$
\end{theorem}

\begin{proof} It follows from (4.2) and Theorem \ref{th:7.5'}.
\end{proof}

\section{Conclusion}

We have extended the study of representing  states on  effect
algebras by integrals that was started in the paper \cite{Dvu2} for
states on pseudo effect algebras, quantum structures where the
partial addition is not more assumed to be commutative.

Our research is based on methods of simplices and their application
to state spaces. We have showed that every pseudo effect algebra
that satisfies the same kind of the Riesz Decomposition Property,
(RDP),  is always a Choquet simplex, Theorem \ref{th:4.2}.  This
Theorem extends the result known for effect algebras with (RDP), see
\cite[Thm 5.1]{Dvu1}, for pseudo effect algebras with a stronger
version, (RDP)$_1$, that is always an interval in a unital po-group
with (RDP)$_1,$ and for interval pseudo effect algebras with (RDP),
see \cite[Thm 4.3]{Dvu3}. We note that we do not know whether every
pseudo effect algebra with (RDP) is an interval in a unital
po-group.

Finally, this result was applied to represent states on pseudo
effect algebras with (RDP) by integrals through  regular Borel
probability measures, Theorem \ref{th:7.3'} and Theorem
\ref{th:7.5'}.

It is important to make a finale remark that formulas (5.1) and
(5.2) show that they are a bridge between the approach by de Finetti
who was a propagator of probabilities as finitely additive measures,
and the approach by Kolmogorov for whom  a probability measure was a
$\sigma$-additive measure, \cite{Kol}. The mentioned formulas say by
a way that these two approaches are equivalent.

\end{document}